\documentclass[12pt]{article}
\usepackage{geometry}    
\usepackage{graphicx}
\usepackage{amssymb}
\usepackage{epstopdf}
\usepackage{amsthm}

\newtheorem{theorem}{Theorem}
\newtheorem{lemma}[theorem]{Lemma}

\title{A counterexample to a result on the \\ tree graph of a graph\footnote{Partially supported by CONACyT  M\'exico, projects 169407 and 178910.}}

\author{ANA PAULINA FIGUEROA  \\ {\it Departamento de Matem\'aticas }\\\textit{ Instituto Tecnol\'ogico Aut\'onomo de M\'exico} \\ {\tt ana.figueroa@itam.mx}\\ \\ EDUARDO RIVERA-CAMPO  \\ {\it Departamento de Matem\'aticas }\\\ \textit{Universidad Aut\'onoma Metropolitana-Iztapalapa} \\  {\tt erc@xanum.uam.mx} }

\date{}

\begin{document}
\maketitle

\begin{abstract}
Given a set of cycles $C$ of a graph $G$, the tree graph of $G$ defined by $C$ is the graph $T(G,C)$ whose vertices are the spanning trees of $G$ and in which two trees $R$ and $S$ are adjacent if $R \cup S$ contains exactly one cycle and this cycle  lies in $C$. Li et al [Discrete Math 271 (2003), 303--310] proved that if the graph $T(G,C)$ is connected, then $C$ cyclically spans the cycle space of $G$.  Later, Yumei Hu [Proceedings of the 6th International Conference on Wireless Communications Networking and Mobile Computing (2010), 1--3] proved that if $C$ is an arboreal family of cycles of $G$ which cyclically spans the cycle space of a $2$-connected graph $G$, then $T(G, C)$ is connected.  In this note we present an infinite family of counterexamples to Hu's result. 
\end{abstract}

\section{Introduction}

The \emph {tree graph} of a connected graph $G$ is the graph $T(G)$ whose vertices are the spanning trees of $G$, in which two trees $R$ and $S$ are adjacent if $R \cup S$ contains exactly one cycle. Li et al \cite{LNR}  defined the \emph {tree graph of $G$ with respect to a set of cycles $C$} as the spanning subgraph $T(G,C)$ of $T(G)$ where two trees $R$ and $S$ are adjacent only if  the unique cycle contained in $R \cup S$ lies in $C$. 

A set of cycles $C$ of $G$ \emph{cyclically spans}  the cycle space of $G$ if for each cycle $\sigma$ of $G$ there are cycles $\alpha_1, \alpha_2, \ldots, \alpha_m \in C$ such that: $\sigma = \alpha_1 \Delta \alpha_2 \Delta \dots \Delta \alpha_m$ and, for $i =2, 3, \ldots, m$, $\alpha_1 \Delta \alpha_2 \Delta \dots \Delta \alpha_i$ is a cycle of $G$. Li et al \cite{LNR} proved the following theorem:

\begin{theorem} \label{Li1} If $C$ is a set of cycles of a connected graph $G$ such that the graph $T(G,C)$ is connected, then $C$ cyclically spans the cycle space of $G$.  
\end{theorem}

A set of cycles $C$ of a graph $G$ is  \emph{arboreal} with respect to $G$ if for every spanning tree $T$ of $G$, there is a cycle $\sigma \in C$ which is a fundamental cycle of $T$. Yumei Hu  \cite{H}  claimed to have proved the converse theorem:

\begin{theorem} \label{Hu}
Let $G$ be a $2$-connected graph. If $C$ is an arboreal set of cycles of $G$ that cyclically spans the cycle space of $G$, then $T(G,C)$ is connected.
\end{theorem}

In this note we present a counterexample to Theorem \ref{Hu} given by a triangulated plane graph $G$ with 6 vertices and an arboreal family of cycles $C$ of $G$ such that $C$ cyclically spans the cycle space of $G$, while $T(G, C)$ is disconnected.  Our example generalises to a family of triangulated graphs $G_n$ with $3(n+2)$ vertices for each integer $n\geq 0$.

If $\alpha$ is a face of a plane graph $G$, we denote, also by $\alpha$, the corresponding cycle of $G$ as well as the set of edges of $\alpha$.

\section{Preliminary results}

Let $G$ be a plane graph. For each cycle $\tau$, let $k(\tau)$ be the number of faces of $G$ contained in the interior of $\tau$.  A \emph{diagonal edge} of $\tau$ is an edge lying in the interior of $\tau$ having both vertices in $\tau$. The following lemma will be used in the proof of Theorem \ref{cyclically}.

\begin{lemma} \label{twocycles}
Let $G$ be a triangulated plane graph and $\sigma$ be a cycle of $G$. If $k(\sigma) \geq 2$, then there are two faces $\phi$ and $\psi$ of $G$, contained in the interior of $\sigma$, both with at least one edge in common with $\sigma$, and such that $\sigma \Delta \phi$ and $\sigma \Delta \psi$ are cycles of $G$.
\end{lemma}

\begin{proof}

If $k(\sigma)=2$, let $\phi$ and $\psi$ be the two faces of $G$ contained in the interior of $\sigma$. Clearly $\sigma \Delta \phi = \psi$  and $\sigma \Delta \psi=\phi$ which are cycles of $G$. 

Assume $k=k(\sigma) \geq 3$ and that the result holds for each cycle $\tau$ of $G$ with $2 \leq k(\tau) < k$. If $\sigma$ has a diagonal edge $uv$, then $\sigma$ together with the edge $uv$ define two cycles $\sigma_1$ and $\sigma_2$ such that $k(\sigma)=k(\sigma_1)+k(\sigma_2)$. If $\sigma_1$ is a face of $G$, then 
$\sigma \Delta \sigma_1$ is a cycle of $G$ and, if $k(\sigma_1) \geq 2$, then by induction there are two faces 
$\phi_1$ and $\psi_1$ of $G$, contained in the interior of $\sigma_1$, both with at least one edge in common with $\sigma_1$, and such that $\sigma_1 \Delta \phi_1$ and $\sigma_1 \Delta \psi_1$ are cycles of $G$. Without loss of generality, we assume $uv$ is not an edge of $\phi_1$ and therefore $\phi=\phi_1$ has at least one edge in common with $\sigma$ and is such that $\sigma \Delta \phi$ is a cycle of $G$.  Analogously $\sigma_2$  contains a face $\psi$ with at least one edge in common with $\sigma$ and such that $\sigma \Delta \psi$ is a cycle of $G$.

For the remaining of the proof we may assume $\sigma$ has no diagonal edges. Since $k=k(\sigma) \geq 3$, there is a vertex $u$ of $\sigma$ which is incident with one or more edges lying in the interior of $\sigma$. Let $v_0, v_1, \ldots, v_m$ be the vertices in $\sigma$ or in the interior of $\sigma$ which are adjacent to $u$. Without loss of generality we assume $v_0$ and $v_m$ are vertices of $\sigma$ and that $v_0,v_1, \ldots, v_{m-1}, v_m$ is a path joining $v_0$ and $v_m$, see Figure \ref{fig:dentro}.

\begin{figure}[ht] 
\begin{center}
 \includegraphics[width=.4\textwidth]{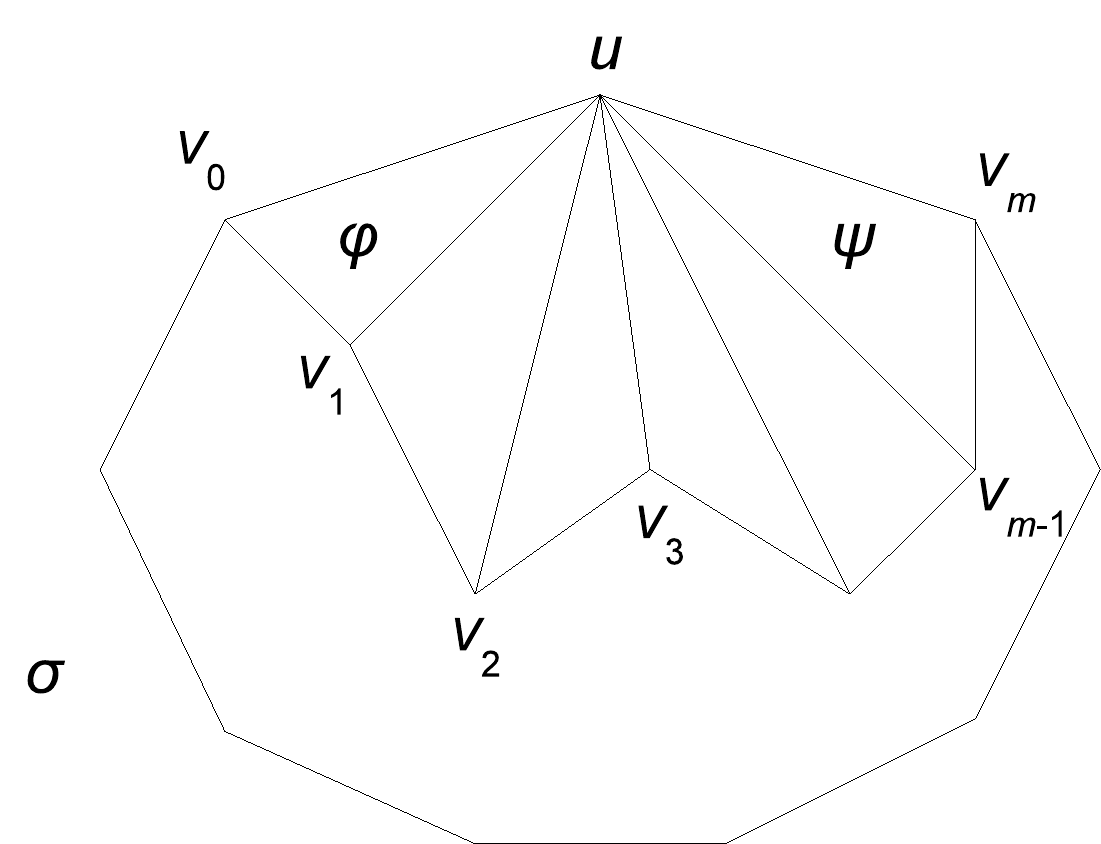}
 \caption{\small  Cycle $\sigma$ with no diagonal edges.}
  \label{fig:dentro} 
\end{center}
\end{figure}

As $\sigma$ has no diagonal edges, vertices $v_1, v_2, \ldots, v_{m-1}$ are not vertices of $\sigma$ and therefore faces $\phi = uv_0v_1$ and $\psi = uv_mv_{m-1}$ are such that $\sigma \Delta \phi$ and $\sigma \Delta \psi$ are cycles of $G$, each with one edge in common with $\sigma$.

\end{proof}

\begin{theorem} \label{cyclically}
Let $G$ be a triangulated  plane graph and $\alpha$ and $\beta$ be two internal faces of $G$ with one edge in common. If $C$ is the set of internal faces of $G$ with cycle $\alpha$ replaced by the cycle $\alpha \Delta \beta$, then $C$ cyclically spans the cycle space of $G$. 
\end{theorem}

\begin{proof}
Let $\sigma$ be a cycle of $G$. If $k(\sigma) =1$, then $\sigma \in C$ or $\sigma = \alpha$ in which case $\sigma = (\alpha \Delta \beta) \Delta \beta$. In both cases $\sigma$ is cyclically spanned by $C$. 

We proceed by induction assuming $k = k(\sigma) \geq 2$ and that if $\tau$ is a cycle of $G$ with $k(\tau) <  k$, then $\tau$ is cyclically spanned by $C$.

By Lemma \ref{twocycles}, there are two faces $\phi$ and $\psi$ of $G$, contained in the interior of $\sigma$, such that both $\sigma \Delta \phi$ and $\sigma \Delta \psi$ are cycles of $G$; without loss of generality we assume $\phi \neq \alpha$. Clearly $k(\sigma \Delta \phi) < k$;  by induction, there are cycles $\tau_1, \tau_2, \ldots, \tau_m \in C$ such that: $\sigma \Delta \phi = \tau_1 \Delta \tau_2 \Delta \dots \Delta \tau_m$ and, for $i =2, 3, \ldots, m$, $\tau_1 \Delta \tau_2 \Delta \dots \Delta \tau_i$ is a cycle of $G$. As $\sigma = (\sigma \Delta \phi) \Delta \phi = \tau_1 \Delta \tau_2 \Delta \dots \Delta \tau_m \Delta \phi$, cycle $\sigma$ is also cyclically spanned by $C$. 
\end{proof}

\section{Main result}

Let $G$ be the skeleton graph of a octahedron (see Figure \ref{fig:G'0})  and $C$ be the set of cycles that correspond to the internal faces of $G$ with cycle $\alpha$ replaced by cycle $\alpha \Delta \beta$.

\begin{figure}[ht] 
\begin{center}
 \includegraphics[width=.3\textwidth]{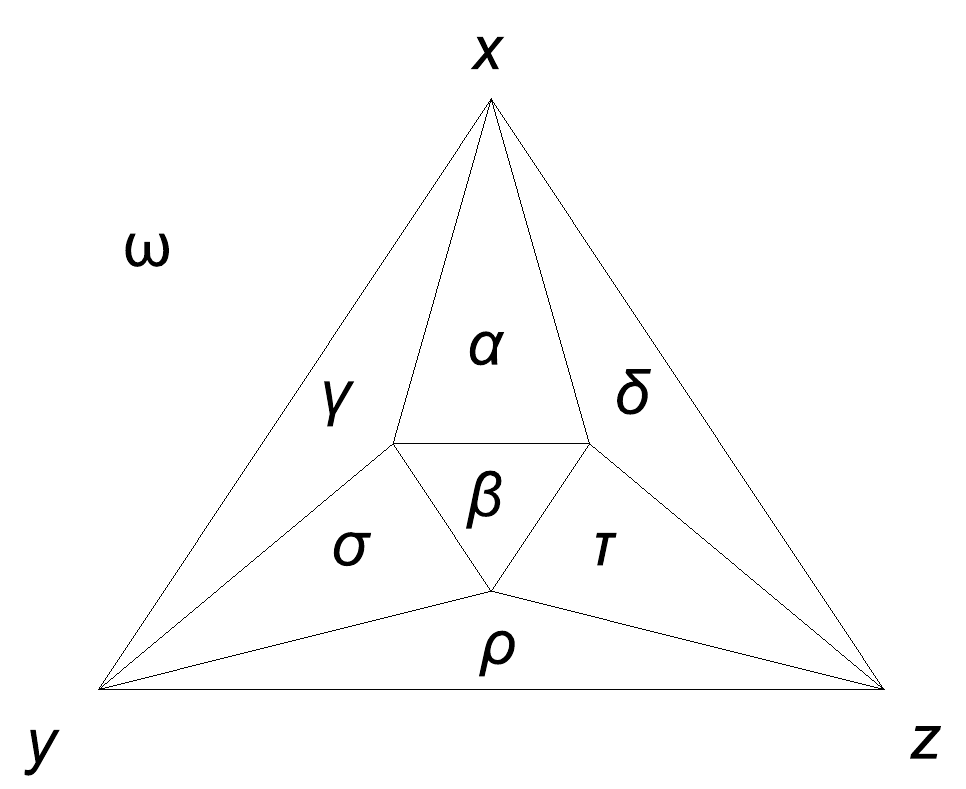}
 \caption{\small Graph $G$ with internal faces $\alpha, \beta, \gamma, \delta, \sigma, \tau$ and $\rho$ and outer face $\omega$.}
  \label{fig:G'0} 
\end{center}
\end{figure}

By Theorem \ref{cyclically}, $C$  cyclically spans the cycle space of $G$. Suppose $C$ is not arboreal and let $P$ be a spanning tree of  $G$ with none of its fundamental cycles  in $C$.  For this to happen, each of the cycles $\beta, \gamma, \delta, \sigma, \tau$ and $\rho$ of $G$, must have at least  two edges which are not edges of $P$ and since $P$ has no cycles, at least one edge of cycle $\alpha$ and at least one edge of cycle $\omega$ are not edgs of $P$. 

Therefore $G$ has at least $7$ edges which are not edges of $P$. These, together with the 5 edges of $P$ sum up to 12 edges which is exactly the number of  edges of $G$. This implies that each of the cycles $\omega$ and $\alpha$ has exactly two edges of $P$ and that each of the cycles $\beta, \gamma, \delta, \sigma, \tau$ and $\rho$ has exactly one edge of $P$. 

If edges $xy$ and $xz$ are edges of $P$, then vertex $x$ cannot be incident to any other edge of $P$ and therefore cycle $\alpha$ can only have one edge of $P$, which is not possible.

If edges $xy$ and $yz$ are edges of $P$, then vertex $y$ cannot be incident to any other edge of $P$. In this case, the edge in cycle $\sigma$, opposite to vertex $y$, must be an edge of $P$ and cannot be incident to any other edge of $P$, which  again is not possible. The case where edges $xz$ and $yz$ are edges of $P$ is analogous. Therefore $C$ is an arboreal set of cycles of $G$.

Let $T$, $S$ and $R$ be the spanning trees of $G$ given in Figure \ref{fig:TRS0}. The graph $T(G, C)$ has a connected component formed by the trees $T$, $S$ and $R$ since cycle $\rho$ is the only cycle in $C$ which is a fundamental cycle of either $T$, $S$ or $R$. This implies that $T(G, C)$ is disconnected. 

\begin{figure}[ht] 
\begin{center}
 \includegraphics[width=.8\textwidth]{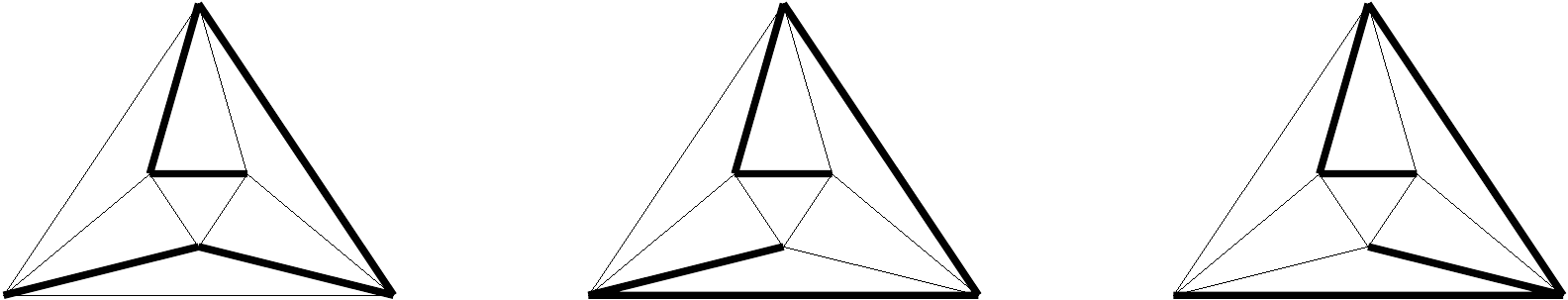}
 \caption{\small Trees $T$ (left), $S$ (center) and $R$ (right).}
  \label{fig:TRS0} 
\end{center}
\end{figure}

We proceed to generalise the counterexample to graphs with arbitrary large number of vertices. Let $G_0 = G, x_0=x, y_0=y, z_0=z$ and for $t \geq 0$ define $G_{t+1}$  as the graph obtained  by placing a copy of $G_t$ in the inner face of the skeleton graph of an octahedron as in Figure  \ref{fig:G'(t+1)}.

\begin{figure}[ht] 
\begin{center}
 \includegraphics[width=.5\textwidth]{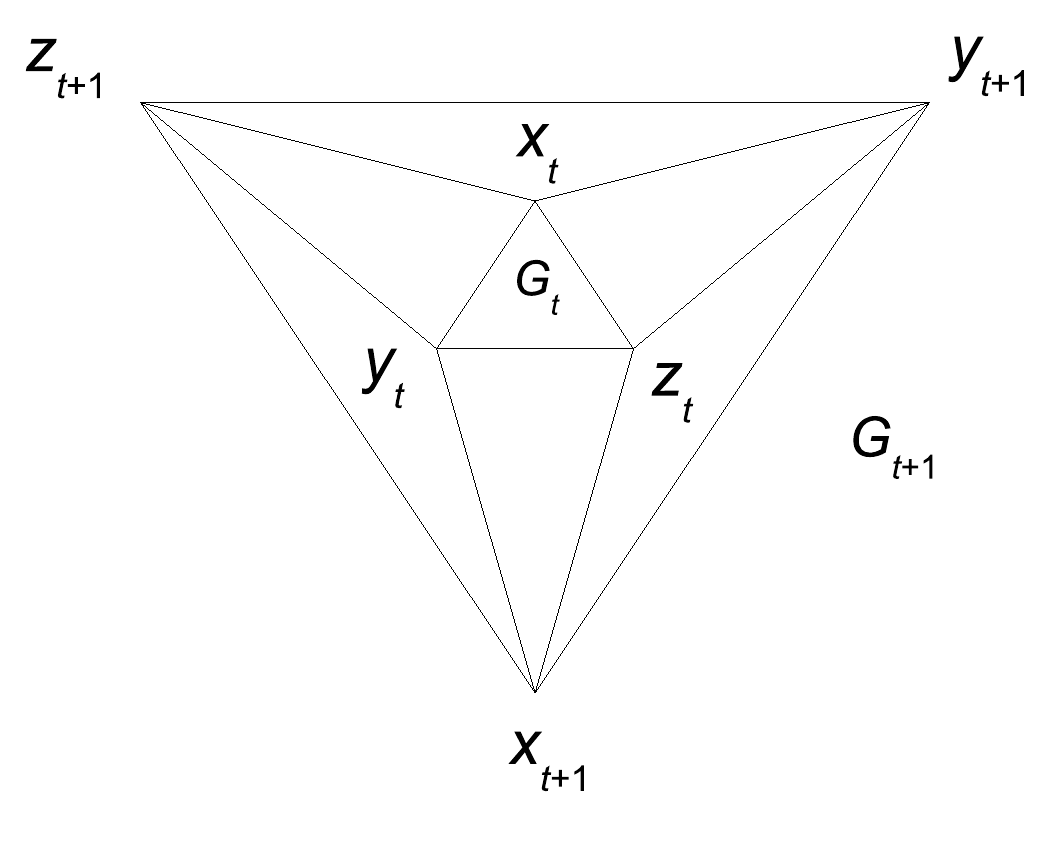}
 \caption{\small $G_{t+1}$.}
 \label{fig:G'(t+1)} 
\end{center}
\end{figure}

Notice that each graph $G_n$ contains a copy $G'$ of $G$ in the innermost layer. We also denote by $\alpha, \beta, \gamma, \delta, \sigma, \tau$ and $\rho$ the cycles of $G_n$ that correspond to the cycles $\alpha, \beta, \gamma, \delta, \sigma, \tau$ and $\rho$ of $G'$. Let $\omega_{n}$ denote the cycle given by the edges in the outer face of $G_{n}$. 

For $n  \geq 0$ let $C_n$ be the set of cycles that correspond to the internal faces of $G_n$ with cycle $\alpha$ replaced by cycle $\alpha \Delta \beta$. By Theorem \ref{cyclically}, $C_n$  cyclically spans the cycle space of $G_n$. 

We claim that for $n \geq 0$, set $C_n$ is an arboreal set of cycles of $G_n$. Suppose $C_{t}$ is arboreal but  $C_{t+1}$ is not and let $P_{t+1}$ be a spanning tree of $G_{t+1}$ such that none of its fundamental cycles lies in $C_{t+1}$. 

As in the case of graph $G$ and tree $P$, above, each cycle in $C_{t+1}$, other than  $\alpha \Delta \beta$,  has exactly one edge in $P_{t+1}$,  while cycles $\alpha$ and $\omega_{t+1}$ have exactly two edges in $P_{t+1}$. The reader can see that this implies that the edges of $P_{t+1}$ which are not edges of $G_t$ form a path with length 3. Then $P_{t+1}-\left\{x_{t+1}, y_{t+1}, z_{t+1} \right\}$ is a spanning tree of $G_t$ and by induction, one of its fundamental cycles lies in $C_t \subset C_{t+1}$ which is a contradiction. Therefore $C_{t+1}$ is arboreal.

Let $T_0 = T$ and for $t \geq 0$ define $T_{t+1}$ as the spanning tree of $G_{t+1}$ obtained by placing a copy of $T_t$ in the inner face of the skeleton graph of an octahedron as in Figure \ref{fig:T'(t+1)}.

\begin{figure}[ht] 
\begin{center}
 \includegraphics[width=1\textwidth]{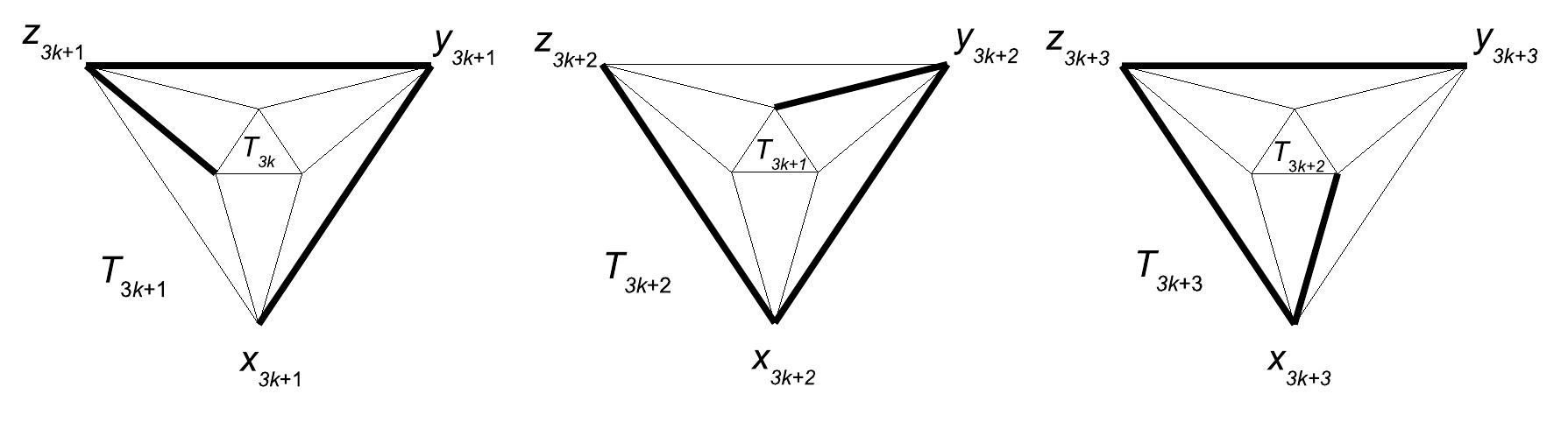}
 \caption{\small $t=3k$ (left), $t=3k+1$ (centre) and $t=3k+2$ (right).}
  \label{fig:T'(t+1)} 
\end{center}
\end{figure}

Trees $S_{t+1}$ and $R_{t+1}$ are obtained from $S_t$ and $R_t$ in the same way with $S_0=S$ and $R_0=R$ respectively. We claim that, for each integer $n \geq 0$, cycle $\rho$ is the only cycle in $C_{n}$ which is a fundamental cycle of either $T_{n}, S_{n}$ or $R_{n}$. Therefore $T_{n}, S_{n}$ and $R_{n}$ form a connected component of $G_{n}$ which implies that  $T(G_{n}, C_{n})$ is  disconnected. 



\begin{thebibliography}{9}

\bibitem{H} 
Hu, Y.:  
\newblock{A necessary and sufficient condition on the tree graph defined by a set of cycles},  
\newblock{In Proceedings of the 6th International Conference on Wireless Communications Networking and Mobile Computing (WiCOM)} 
\newblock{(2010)}, 
\newblock{pp. 1 -- 3},
\newblock{IEEE}.

\bibitem{LNR} 
Li, X., Neumann-Lara, V., Rivera-Campo, E.:
\newblock {On the tree graph defined by a set of cycles}, 
\newblock {Discrete Math.}
\newblock {271}
\newblock {(2003)}, 
\newblock {303--310}.


\end{thebibliography}
\end{document}